\documentclass[leqno,12pt]{amsart} 
\usepackage{amssymb} 
\usepackage[all]{xy}
\usepackage{graphicx}

\theoremstyle{plane} 
\newtheorem{theorem}{\indent\sc Theorem}[section] 
\newtheorem{lemma}[theorem]{\indent\sc Lemma}
\newtheorem{corollary}[theorem]{\indent\sc Corollary}
\newtheorem{proposition}[theorem]{\indent\sc Proposition}

\theoremstyle{definition} 
\newtheorem{definition}[theorem]{\indent\sc Definition}
\newtheorem{remark}[theorem]{\indent\sc Remark}
\newtheorem{example}[theorem]{\indent\sc Example}

\begin{document}

\title[Projective bundles with nef normalized tautological divisor]{On projective space bundles with nef normalized tautological divisor} 

\author[K. Yasutake]{Kazunori Yasutake} 


\address{
Graduate School of Mathematical Sciences \endgraf
Kyushu University \endgraf
Fukuoka 819-0395 \endgraf
Japan
}
\email{k-yasutake@math.kyushu-u.ac.jp}

\maketitle

\begin{abstract}
In this paper, we study the structure of projective space bundles whose normalized tautological divisor is nef.
As an application, we get a characterization of abelian varieties up to finite \'etale covering.
\end{abstract}

\section*{Introduction}
For a morphism between smooth projective varieties $\pi : Y \rightarrow X$ the relative anti-canonical divisor $-K_{\pi}$ on Y is defined by the difference of anticanonical divisors
$-K_{\pi}:=-K_Y-\pi^*(-K_X)$. J. Koll\'ar, Y. Miyaoka and S. Mori proved that the relative anti-canonical divisor of a non-constant generically smooth morphism cannot be ample 
in arbitrary characteristic \cite{kollar-miyaoka-mori}, \cite{miyaoka1}. 
In the case where $\pi : Y=\mathbb{P}_X(\mathcal{E}) \rightarrow X$ is a projectivization of vector bundle on X, we know that the relative anti-canonical divisor is positive proportion of the normalized tautological divisor. 
Miyaoka studied the case where Y is a curve and showed that the nefness of the normalized tautological divisor is equal to the semistability of vector bundle \cite{miyaoka2}. 
Nakayama generalized this to the arbitrary dimension in \cite{nakayama2}. In this paper we study the more explicit structure of vector bundles with nef normalized tautological divisor. 
In Section 1, we review the definition and some known results. 
In Section 2, we treat semiample cases and show that a pullback of such a bundle by some finite covering is trivial up to twist by some line bundle.    
In Section 3, we treat the case where the normalized tautological divisor is nef. In particular in positive characteristic these bundles are charactrized as special strongly semistable
bundles.
In Section 4 we study manifolds whose tangent bundle have a nef normalized tautological divisor. We prove such surfaces are isomorphic to a quotient of abelian surface by some finite \'etale morphism. Moreover under the assumption that such a divisor is semiample, we can show that finite \'etale covering of abelian varieties are all varieties satisfying this property.

\section*{Acknowledgements}
The author would like to express his gratitude to his supervisor Professor Eiichi Sato for many useful discussions and much warm encouragement. He would also like to thank Professor Noboru Nakayama for helpful advice. 

\section*{notation}
We will work, throughout this paper, over the complex number field $\mathbb{C}$ unless otherwise mentioned.
We freely use the customary terminology in algebraic geometry.
Vector bundles are often identified with the locally free sheaves of their sections, and these words are used interchangeably.
Line bundles are identified with linear equivalence classes of Cartier divisors , and their tensor products are denoted additively.

\section{Preliminary}
At first we recall the notion of normalized tautological divisor defined in \cite{nakayama2} which is originally introduced by Y. Miyaoka as normalized hyperplane class in \cite{miyaoka2}.

\begin{definition}
Let X be a smooth projective variety over algebraic closed field of arbitrary characteristic and $\mathcal{E}$ a rank $r$ vector bundle on X. 
A normalized tautological divisor $\Lambda_{\mathcal{E}}$ of $\mathbb{P}_X(\mathcal{E})$ is a $\mathbb{Q}$-divisor on $\mathbb{P}_X(\mathcal{E})$
such that the equality $r \Lambda_{\mathcal{E}}=r \xi_{\mathcal{E}}- \pi ^{*}(det(\mathcal{E}))$ holds where $\xi_{\mathcal{E}}$ is the tautological divisor on 
$\mathbb{P}_X(\mathcal{E})$ and 
$\pi$ is the natural projection $\pi: \mathbb{P}_X(\mathcal{E})\rightarrow X$. In particular, $r \Lambda_{\mathcal{E}}$ is linearly equivalent                                                                                                                                              to the relative anti-canonical divisor $-K_{\pi}:=-K_{\mathbb{P}_X(\mathcal{E})}-\pi^*(-K_X)$.
\end{definition}

By virtue of the following theorem, we know that the normalized tautological divisor cannot be ample.

\begin{theorem}$($\cite{kollar-miyaoka-mori},\cite{miyaoka1},\cite{zhang},\cite{debarre}, char\;$\geqslant0)$
Let X and Y be smooth projective varieties over an algebraically closed field of arbitrary charactristic and let $\pi : X \rightarrow Y$ be a nonconstant generically smooth morphism.
Let H be an ample divisor on Y. 
For any positive $\epsilon$, the divisor $-K_{\pi}-\epsilon \pi^{*}H$ is not nef. In particular, the relative anti-canonical divisor $-K_{\pi}$ is not ample.
\end{theorem}

Moreover N.Nakayama showed that the normalized tautological divisor of a projective space bundle $\pi: \mathbb{P}_M(\mathcal{E})\rightarrow M$
cannot be nef and big in characteristic 0. To state his theorem, we recall the definition of numerical D-dimension of nef divisors on smooth projective varieties.

\begin{definition}
Let X be an n-dimensional smooth projective variety, A an ample divisor and D a nef divisor on X. 
We define the numerical D-dimension $\nu(D, X)$ of D by 
\[\nu(D, X)=\displaystyle \max \{k\in \mathbb{N} \;| \;D^k \cdot A^{n-k} \not = 0\}.\]
We call D big if $\nu(D, X)=n$.
\end{definition}

\begin{theorem}[\cite{nakayama2}]\label{section}
Let X be a smooth complex projective variety and $\mathcal{E}$ a rank r vector bundle on X. 
Assume that the relative anti-canonical divisor $-K_{\pi}=r\Lambda_{\mathcal{E}}$ of projective space bundle $\pi: \mathbb{P}_X(\mathcal{E})\rightarrow X$
is nef. 
Then $\nu(-K_{\pi})=r-1$.
In particular the normalized tautological divisor cannot be nef and big.
\end{theorem}

In positive characteristic case we can also have a similar statement.

\begin{proposition}$($char$\;>0)$
Let X be a smooth projective variety over an algebraic closed field of positive characteristic and let $\mathcal{E}$ be a rank r vector bundle on X. 
Assume that the normalized tautological divisor $\Lambda_{\mathcal{E}}$ is nef. Then $\Lambda_{\mathcal{E}}$ cannot be big.
\end{proposition}

\begin{proof}
If $\Lambda_{\mathcal{E}}$ is big, we consider the diagram 
\[\xymatrix{
 \mathbb{P}_C(f^*\mathcal{E})\ar[rr]^{\bar{f}} \ar[d]^/1mm/{\pi|_C}  & & \mathbb{P}_X(\mathcal{E}) \ar[d]^/1mm/{\pi}\\
 C \ar[rr]^f&&  X\\
}\] 
for a general curve $f: C\rightarrow X$. Then $\Lambda_{f^*\mathcal{E}}=\bar{f}^*\Lambda_{\mathcal{E}}$ is also nef and big.
But we have $(r\Lambda_{f^*\mathcal{E}})^r=(r \xi_{f^*\mathcal{E}}- \pi|_C ^{*}(det(\mathcal{f^*E})))^r=r^r\xi_{f^*\mathcal{E}}^r+r^r\xi_{f^*\mathcal{E}}^{r-1}\cdot \pi|_C ^{*}(det(\mathcal{f^*E}))=0$. This contradicts the bigness of $\Lambda_{f^*\mathcal{E}}$.
\end{proof}

\section{vector bundles with semiample normalized tautological divisor}

First we consider the case where the normalized tautological divisor of projective bundle $\pi: \mathbb{P}_X(\mathcal{E})\rightarrow X$ is semiample. 
Such a case really exists. For example let $\mathcal{E} \cong \mathcal{O}^r_X$ be a trivial vector bundle on smooth projective variety X,  the relative anti-canonical divisor $-K_{\pi}=p^{*}(-K_{\mathbb{P}^{r-1}})$ is basepoint-free where $p :  \mathbb{P}_X(\mathcal{E})\cong X \times \mathbb{P}^{r-1} \rightarrow \mathbb{P}^{r-1}$ is the second projection.
Therefore the normalized tautological divisor is also nef. 
In the case where dimX$=1$, N. Nakayama showed the following theorem. 

\begin{theorem}[\cite{nakayama1}]
Let $\pi : X=\mathbb{P}_C(\mathcal{E}) \rightarrow C$ be a $\mathbb{P}^1$-bundle over a smooth curve C.
Then following two conditions are equivalent:
 \begin{enumerate}
\item there is a finite \'etale morphism $f: \tilde{C} \rightarrow C$ such that $X \times _{C} \tilde{C} \cong \mathbb{P}^1 \times \tilde{C} $ over $\tilde{C}$;
\item the normalized tautological line bundle $\Lambda_{\mathcal{E}}$ is semiample.
\end{enumerate}
\end{theorem}

\begin{remark}
Forthermore N. Nakayama proved very interesting fact; for $\mathbb{P}^1$-bundle $\mathbb{P}_C(\mathcal{E})$ over a smooth curve C of genus $g(C)>1$ following two conditions are equivalent:
\begin{enumerate}
\item the normalized tautological divisor $\Lambda_{\mathcal{E}}$ is semiample.
\item X has a surjective endomorphism $g: X \rightarrow X$ that is not isomorphism.
\end{enumerate}
\end{remark}

We generalize the result of N. Nakayama to arbitrary rank vector bundles and arbitrary dimensional base manifolds.

\begin{theorem}\label{semiample}
Let X be a smooth projective variety and $\mathcal{E}$ a rank r vector bundle on X. 
Assume that the normalized tautological divisor $\Lambda_{\mathcal{E}}$ of projective space bundle $\pi: Y=\mathbb{P}_X(\mathcal{E})\rightarrow X$
is semiample. 
Then there exist a finite \'etale morphism $f : X^{\prime} \rightarrow X$ such that $f^*\mathcal{E}$ is trivial up to twist by some line bundle.
\end{theorem}

\begin{proof}
From the semiampleness of $\Lambda_{\mathcal{E}}$ we have a fibration $\varphi:Y \rightarrow Z\subseteq \mathbb{P}^N$ defined by the basepoint-free divisor 
$m\Lambda_{\mathcal{E}}$ for $m\gg0$ and 
the Stein factorization. 
\[\xymatrix{
Y=\mathbb{P}_{X}(\mathcal{E}) \ar[rr]_{\varphi} \ar[d]_{\pi}  & &Z\subseteq\mathbb{P}^N \\
{X}   \\
}\]
By Theorem \ref{section} we have dim\;Im$\varphi(Y)=r-1$. 
We may asuume that Z is smooth by the replacement Z with $\mathbb{P}^N$. 
Let $S=\varphi^{-1}(z)$ be a general fiber of $\varphi$ then S is smooth and $\pi|_S:S \rightarrow X$ is finite surjective. 
We will show that $\pi|_S$ is unramified i.e. $\Omega_S\cong\pi^*\Omega_X|_S$.  
In this situation we have the morphism $\Phi:\mathcal{T}_{\pi} \rightarrow \mathcal{T}_Y\rightarrow \varphi^*\mathcal{T}_Z$.
By the restriction of this morphism to S we have the generically injective morphism $\mathcal{T}_{\pi}|_S\rightarrow \varphi^*\mathcal{T}_Z|_S\cong \mathcal{O}_S^{N}$
since $\varphi$ is generically smooth.
Taking some direct summand we have a generically injective morphism between vector bundles of same rank $\mathcal{T}_{\pi}|_S\rightarrow \mathcal{O}_S^{r-1}$.
From the semiampleness of $\Lambda_{\mathcal{E}}$ we know that the determinant morphism $-K_{\pi}|_S\rightarrow \mathcal{O}_S$ is isomorphism. 
Therefore we get an isomorphism $\mathcal{T}_{\pi}|_S\cong \mathcal{O}_S^{r-1}$.
Taking the dual we also have an isomorphism $\Omega_{\pi}|_S\cong \mathcal{O}_S^{r-1}$. 
Next we consider the exact sequence 
\[0 \rightarrow (\pi|_S)^*\Omega_X \rightarrow \Omega_S \rightarrow \Omega_{\pi|_S} \rightarrow 0.\]
We will show that $\Omega_{\pi|_S}=0$.
From the argument mentioned above we know the morphism $\Phi^{\vee}:\varphi^*\Omega_Z|_S\cong \mathcal{O}_S^N \rightarrow \Omega_Y|_S\rightarrow \Omega_{\pi}|_S\cong \mathcal{O}_S^{r-1}$ is surjective.
Since S is smooth subscheme of Y we have an exact sequence 
 \[0 \rightarrow \mathcal{N}_{S/Y}^{\vee} \rightarrow \Omega_Y|_S \rightarrow \Omega_S \rightarrow 0.\]
 The image of $\varphi^*\Omega_Z|_S$ in $\Omega_Y|_S$ is contained in $\mathcal{N}_{S/Y}^{\vee}$ since the image of $\varphi^*\Omega_Z|_S$ in $\Omega_S$ is 0. 
 Therefore we have a surjection $\mathcal{N}_{S/Y}^{\vee}\rightarrow \Omega_{\pi}|_S \rightarrow 0.$
From the following commutative diagram we have $\Omega_{\pi|_S}=0$.
\[\xymatrix{
0 \ar[rr]_{} & & \mathcal{N}_{S/Y}^{\vee} \ar[rr] \ar[d]  & & \Omega_S \ar[d]\\
& &  \Omega_{\pi}|_S \ar[rr] \ar[d] &&   \Omega_{\pi|_S} \ar[rr] \ar[d] & &0\\
&& 0 && 0\\
}\] 
Therefore we know that $\pi|_S:S \rightarrow X$ is \'etale.
Taking a base-change by the \'etale morphism $\pi|_S:S \rightarrow X$, we have a finite \'etale morphism $g: \mathbb{P}_S(\pi|_S^*(\mathcal{E}))\rightarrow \mathbb{P}_X(\mathcal{E})$ and the natural projection $\tilde{\pi}: \mathbb{P}_S(\pi|_S^*(\mathcal{E}))\rightarrow S$ has a smooth section.
By using this argument repeatedly, we have a finite \'etale morphism $f : X^{\prime} \rightarrow X$ such that the natural projection 
$\pi^{\prime}:\mathbb{P}_{X^{\prime}}(f^*(\mathcal{E}))\rightarrow X^{\prime}$ has sufficiently many disjoint sections. Therefore we can show
that $\mathbb{P}_{X^{\prime}}(f^*(\mathcal{E}))\cong X^{\prime}\times \mathbb{P}^{r-1}$ and we have $f^*\mathcal{E}$ is trivial up to twist by a line bundle.
\end{proof}

By the similar way we can prove the following Proposition in positive characteristic.

\begin{proposition}$($char$\;>0)$
Let X be a smooth projective variety over algebraic closed field of positive characteristic and $\mathcal{E}$ a rank r vector bundle on X. 
Assume that the normalized tautological divisor of projective space bundle $\pi: Y=\mathbb{P}_X(\mathcal{E})\rightarrow X$
is semiample. 
Then there exist a finite morphism $f : X^{\prime} \rightarrow X$ from normal variety such that $f^*\mathcal{E}$ is trivial up to twist by a line bundle.
\end{proposition}

\section{vector bundles with nef normalized tautological divisor}

Next we consider the case where the normalized tautological divisor $\Lambda_{\mathcal{E}}$ is nef.  
To state Theorems proved by Y.Miyaoka and N.Nakayama we review the definition of stability of vector bundles.

\begin{definition}
Let $\mathcal{E}$ be a vector bundle on smooth projective variety X of dimension n and A an ample line bundle on X.
$\mathcal{E}$ is said to be A-semistable in the sense of Takemoto-Mumford if 
\[\mu(\mathcal{F})\leqslant \mu(\mathcal{E})\]
for every non-zero subsheaf $\mathcal{F} \subset \mathcal{E}$ where $\mu(\mathcal{F}):=c_1(\mathcal{F}).A^{n-1}/rank(\mathcal{F})$.
\end{definition}

\begin{definition}
Let $\mathcal{E}$ be a vector bundle on smooth projective variety X of dimension n and A an ample line bundle on X.
We denote the absolute Frobenius morphism by $F_X$.
Then $\mathcal{E}$ is said to be strongly A-semistable if $(F_X)^{(n)}(\mathcal{E})$ is A-semistable for every $n\in\mathbb{N}$
where $(F_X)^{(n)}(\mathcal{E})$ is the n-th pullback by absolute Frobenius morphism.
\end{definition}

\begin{theorem}$($\cite{miyaoka2}, char\;$\geqslant0)$\label{miyaoka}
Let $\mathcal{E}$ be a rank r vector bundle on smooth projective curve C over a field k of characteristic $p\geqslant0$. 

$(1)$ If $p=0$, $\Lambda_{\mathcal{E}}$ is nef if and only if $\mathcal{E}$ is $\mu$-semistable.

$(2)$ If $p>0$, $\Lambda_{\mathcal{E}}$ is nef if and only if $\mathcal{E}$ is strongly $\mu$-semistable.

\end{theorem}

\begin{theorem}[\cite{nakayama2}]\label{nakayama}
Let $\mathcal{E}$ be a rank r vector bundle on smooth complex projective variety X of dimension d. Then the following conditions are equivalent:
\begin{enumerate}
\item $\Lambda_{\mathcal{E}}$ is nef;
\item $\mathcal{E}$ is $\mu$-semistable and 
\[(c_2(\mathcal{E})-\displaystyle \frac{2r}{r-1} c_1^2(\mathcal{E})).A^{d-2}=0\]
for an ample divisor A;
\item There is a filtration of vector subbundles
\[0=\mathcal{E}_0\subset \mathcal{E}_1\subset \cdots \subset \mathcal{E}_l=\mathcal{E}\]
 such that $\mathcal{E}_i/ \mathcal{E}_{i-1}$ are projectively flat and the averaged first Chern classes $\mu(\mathcal{E}_i/ \mathcal{E}_{i-1})$ are numerically equivalent to $\mu(\mathcal{E}):=c_1(\mathcal{E})/rank\mathcal{E}$ for any i. 
\end{enumerate}
\end{theorem}

Here, a vector bundle $\mathcal{E}$ is called projectively flat if it admits a projectively flat Hermitian metric. Nakayama shows that projectively flat vector bundles are induced from some representations of the fundamental group of the base space.   

\begin{proposition}[\cite{nakayama2}]\label{representation}
Let $\mathcal{E}$ be a vector bundle of rank r on a smooth complex projective variety X.
Then $\mathcal{E}$ is projectively flat if and only if  the associated $\mathbb{P}^{r-1}$-bundle $\pi : \mathbb{P}_X(\mathcal{E}) \rightarrow X$ is induced from a
projective unitary representation $\pi_1(X)\rightarrow PU(r)$. 
\end{proposition}

Using this fact we immediately get the following result.

\begin{theorem}\label{rc1}
Let X be a smooth projective variety and $\mathcal{E}$ a rank r vector bundle on X.
Assume that X is simply connected and the normalized tautological line bundle $\Lambda=\Lambda_{\mathcal{E}}$ of $\mathcal{E}$ is nef.
Then $\mathcal{E}$ is trivial up to twist by a line bundle.
\end{theorem}

\begin{proof}
On a simply connected manifold we know that there are only trivial projectively flat vector bundles up to twist by Proposition \ref{representation}.
From Theorem \ref{nakayama} $(3)$, we know that $\mathcal{E}$ is constructed by an extension of such vector bundles.
We can denote $\mathcal{E}_1=\mathcal{L}^{\oplus r_1}$ for some line bundle $\mathcal{L}$.
For the bundle $\mathcal{E}_2$ we have an exact sequence 
\[0 \rightarrow \mathcal{E}_1=\mathcal{L}^{\oplus r_1} \rightarrow \mathcal{E}_2 \rightarrow \mathcal{E}_2/ \mathcal{E}_1=\mathcal{M}^{\oplus r_2} \rightarrow 0.\]
By virtue of Theorem \ref{nakayama} $(3)$ we know that the averaged first Chern class $\mu(\mathcal{E}_2)$ is numerical equivalence to $\mu(\mathcal{E}_2/ \mathcal{E}_1)$.
From this we easily get $\mathcal{L}$ is isomorphic to $\mathcal{M}$.  
Since $h^1(\mathcal{O}_X)=0$ the extension above split and we have $\mathcal{E}_2\cong \mathcal{L}^{\oplus r_1+r_2}$.   
Using this argument repeatedly we can show that $\mathcal{E}$ is trivial up to twist by a line bundle.
\end{proof}

\begin{corollary}\label{rc}
Let X is a rationally connected variety and $\mathcal{E}$ a vector bundle on X.  
If the normalized tautological line bundle $\Lambda=\Lambda_{\mathcal{E}}$ of $\mathcal{E}$ is nef then $\mathcal{E}$ is trivial up to twist by a line bundle.
\end{corollary}

In positive characteristic case, we can prove the following results.
This is proved by A. Langer if determinant of $\mathcal{E}$ is trivial.

\begin{theorem}[cf. \cite{langer}, Proposition 5.1]\label{positive}
Let $\mathcal{E}$ be a rank r vector bundle on smooth projective variety X of dimension d in positive characteristic $p>0$. Then the following conditions are equivalent:
\begin{enumerate}
\item $\Lambda_{\mathcal{E}}$ is nef;
\item $\mathcal{E}$ is strongly $A$-semistable and 
\[(c_2(\mathcal{E})-\displaystyle \frac{2r}{r-1} c_1^2(\mathcal{E})).A^{d-2}=0\]
for an ample divisor A. 
\end{enumerate}
\end{theorem}

To show this we use following two lemmata.

\begin{lemma}[\cite{bloch}]
Let X be a smooth projective variety and $\mathcal{L}$ a line bundle on X.
Fix a positive integer $r \in \mathbb{N}$, then there exist a finite surjective morphism $f : X' \rightarrow X$
from smooth variety and line bundle $\mathcal{M}$ on X' such that $f^*(\mathcal{L})\cong r\mathcal{M}$. 
\end{lemma}

\begin{lemma}[c.f. \cite{miyaoka2}, Proposition3.2]
Let $f: Y\rightarrow X$ be a finite separable morphism of smooth projective varieties $\mathcal{E}$ a vector bundle on X and $\mathcal{H}$ an ample line bundle on X.
Then $\mathcal{E}$ is $\mathcal{H}$-semistable if and only if $f^*\mathcal{E}$ is $f^*\mathcal{H}$-semistable.
\end{lemma}

\begin{proof}[Proof of Theorem \ref{positive}]
At first we prove $(1)$ induce (2). Fix an ample divisor $A$ on X. By the theorem above we have a finite surjective morphism $f : X' \rightarrow X$
from smooth variety and line bundle $\mathcal{M}$ on X' such that $f^*(det(\mathcal{E}))\cong r\mathcal{M}$.
Put $\mathcal{E}':=f^*\mathcal{E}(-M)$, then $c_1(\mathcal{E}')=0$ and $\mathcal{E}'$ is nef since $\xi_{\mathcal{E}'}= \bar{f}^* \Lambda_{\mathcal{E}}$ where 
$\bar{f}: \mathbb{P}(\mathcal{E}') \rightarrow  \mathbb{P}(\mathcal{E})$ is a morphism induced by $f$. Therefore every Frobenius pullback $(F_X)^{(n)}(\mathcal{E}')$ is
also nef and any quotient bundle has positive degree with respect to the ample divisor $f^*(A)$. Hence $(F_X)^{(n)}(\mathcal{E}')$ is $f^*(A)$-semistable.
By the argument as in \cite{langer}, Proposition 5.1, we can show the all Chern classes of $\mathcal{E}'$ are numerical trivial with respect to $f^*(A)$.
In particular $(c_2(\mathcal{E}')-\displaystyle \frac{2r}{r-1} c_1^2(\mathcal{E}')).f^*(A)^{n-2}=0.$
Therefore we have $(c_2(f^*\mathcal{E})-\displaystyle \frac{2r}{r-1} c_1^2(f^*\mathcal{E})).f^*(A)^{n-2}$
$=(deg f)^n(c_2(\mathcal{E})-\displaystyle \frac{2r}{r-1} c_1^2(\mathcal{E})).A^{n-2}=0$. Since $f$ is surjective and $f^*\mathcal{E}$ is strongly $f^*(A)$-semistable, we have $\mathcal{E}$ is strongly $A$-semistable.
Next we prove $(2)$ induce (1).
By the theorem above we have a finite surjective morphism $f : X' \rightarrow X$
from smooth variety and line bundle $\mathcal{M}$ on X' such that $f^*(det(\mathcal{E}))\cong r\mathcal{M}$.
More precisely $f$ is  given by a composition of separable morphism and absolute Frobenius morphism.
Therefore $f^*(\mathcal{E})$ is strongly $f^*(A)$-semistable. Hence $\mathcal{E}':=f^*\mathcal{E}(-M)$ is also strongly $f^*(A)$-semistable and $c_1(\mathcal{E}')=0$.
Since $(c_2(\mathcal{E}')-\displaystyle \frac{2r}{r-1} c_1^2(\mathcal{E}')).f^*(A)^{n-2}=0$, we get $c_2(\mathcal{E}').f^*(A)^{n-2}=0$.
By \cite{langer}, Proposition 5.1, we know that  $\mathcal{E}'$ is nef vector bundle. Therefore $\Lambda_{\mathcal{E}}$ is nef since $\xi_{\mathcal{E}'}= \bar{f}^* \Lambda_{\mathcal{E}}$ where 
$\bar{f}: \mathbb{P}(\mathcal{E}') \rightarrow  \mathbb{P}(\mathcal{E})$ is a morphism induced by $f$.
\end{proof}

As a corollary we have an another proof originally showed by A.Langer in \cite{langer} Corollary  5.3.

\begin{corollary}
Let X be a smooth projective variety H an ample divisor. Let $\mathcal{E}$ be a H-strongly semistable vector bundle with $\Delta(\mathcal{E})=0$. 
Then for any smooth closed subvariety $Y\subseteq X$ the restriction $\mathcal{E}|_Y$ is also $H|_Y$-strongly semistable.
\end{corollary}

If the base manifold is a surface we can prove the same statement as in Corollary \ref{rc}.

\begin{theorem}$($char$\;>0)$\label{positive}
let S be a smooth projective rational surface over an algebraically closed field $k$ of positive characteristic $p>0$ and $\mathcal{E}$ a vector bundle of rank r on S.
Assume that the anticanonical bundle of the projection $\pi : \mathbb{P}_S(\mathcal{E}) \rightarrow S$ is nef, then $\mathcal{E}$ is isomorphic to a trivial bundle up to twist.   
\end{theorem}

To show this we prepare following lemmas.

\begin{lemma}$($char$\;\geqslant0)$\label{line}
Let X be a smooth projective variety and $\mathcal{E}$ a vector bundle on X.
Assume that the normalized tautological line bundle $\Lambda=\Lambda_{\mathcal{E}}$ of $\mathcal{E}$ is nef.
Then for any rational curve $\gamma : \mathbb{P}^1 \rightarrow X$ on X, $\gamma^*(\mathcal{E})$ is trivial up to twist by a line bundle.
\end{lemma}

\begin{proof}
We have a splitting :
\[ \gamma ^* \mathcal{E}\cong \mathcal{O}_{\mathbb{P}^1}(a_1)\oplus \mathcal{O}_{\mathbb{P}^1}(a_2)\oplus \cdots \mathcal{O}_{\mathbb{P}^1}(a_r), \]
where $a_1 \leqslant a_2 \leqslant \cdots \leqslant a_r$. 
We have  the morphism $\gamma^{\prime} : \mathbb{P}_{\mathbb{P}^1}(\gamma ^* \mathcal{E}) \rightarrow \mathbb{P}_X(\mathcal{E})$ indeced by $\gamma$.
We denote natural projections by $\pi : \mathbb{P}_X(\mathcal{E}) \rightarrow X$ and $\pi^{\prime} :  \mathbb{P}_{\mathbb{P}^1}(\gamma ^* \mathcal{E}) \rightarrow \mathbb{P}^1$ . 
Then we have
\[r\Lambda_{\gamma ^* \mathcal{E}}=r\xi_{\gamma ^* \mathcal{E}}-{\pi^{\prime}}^*  det(\gamma ^* \mathcal{E})=r\gamma^*\xi_{\mathcal{E}}-\gamma^* \pi^* det(\mathcal{E})=r\gamma^*\Lambda_{\mathcal{E}}.\]
Since $\Lambda_{\mathcal{E}}$ is nef by assumption, $\Lambda_{\gamma ^* \mathcal{E}}$ is also nef. 
Let C be the section of  $\gamma ^* \mathcal{E}$ associated with the quotient line bundle 
$\gamma ^* \mathcal{E} \rightarrow \mathcal{O}_{\mathbb{P}^1}(a_1) \rightarrow 0$.
We have
\[r\Lambda_{\gamma ^* \mathcal{E}}.C=r\xi_{\gamma ^* \mathcal{E}}.C-{\pi^{\prime}}^*  det(\gamma ^* \mathcal{E}).C=ra_1-\displaystyle \sum^{r}_{i=1}a_i\geqslant 0.\]
Therefore we get $a_1=a_2=\cdots =a_r$. Hence $\gamma ^* \mathcal{E}$ is trivial up to twist by a line bundle. 
\end{proof}

\begin{lemma}$($char$\;\geqslant0)$\label{lemma}
Let S be a smooth projective surface $\mathcal{E}$ a rank r vector bundle on S and $f : S^{\prime} \rightarrow S$ a blow-up of S at a point p. 
If $f^*(\mathcal{E})$ is isomorphic to a vector bundle $\mathcal{L}^{\oplus r}$ for some line bundle $\mathcal{L}$ on $S^{\prime}$, then $\mathcal{E}$ is also isomorphic to a vector bundle $\mathcal{M}^{\oplus r}$ for some line bundle $\mathcal{M}$ on S.
\end{lemma}

\begin{proof}
Let C be a exceptional divisor of f. Then $f^*(\mathcal{E})$ is trivial on C i.e. $f^*(\mathcal{E})|_C\cong {\mathcal{L}|_C}^{\oplus r}\cong {\mathcal{O}_C}^{\oplus r}$. 
By the Krull-Schmidt theorem of vector bundle \cite{atiyah}, we have $\mathcal{L}|_C \cong \mathcal{O}_C$. 
Therefore there exists a line bundle $\mathcal{M}$ such that  $f^*(\mathcal{M})\cong \mathcal{L}$. 
Hence we have $f^*(\mathcal{E}\otimes \mathcal{M}^{-1})\cong {\mathcal{O}_{S^{\prime}}}^{\oplus r}$.
Therefore we have $\mathcal{E}\otimes \mathcal{M}^{-1}\cong {\mathcal{O}_{S}}^{\oplus r}$.
\end{proof}

\begin{lemma}$($char$\;\geqslant0)$\label{blowup}
Let $\varphi : X \rightarrow X^{\prime}$ be the blow-up of smooth variety along a smooth subvariety $Y\subset X^{\prime}$ and $\mathcal{E}$
a vector bundle of rank r on X. 
Assume that $\Lambda_{\mathcal{E}}$ is nef. Then there exists a vector bundle $\mathcal{E}^{\prime}$ such that $\mathcal{E}\otimes \mathcal{L}=\varphi^*\mathcal{E}^{\prime}$
where $\mathcal{L}$ is a line bundle on X. 
Moreover $\Lambda_{\mathcal{E}^{\prime}}$ is nef.
\end{lemma}

To prove this we use the following results.

\begin{theorem}$($\cite{andreatta-occhetta}, char$\;\geqslant0)$\label{andreatta}
Let $\varphi : X \rightarrow X^{\prime}$ be the blow-up of smooth variety along a smooth subvariety $Y\subset X^{\prime}$ and $\mathcal{E}$
a vector bundle on X. 
Assume that $\mathcal{E}$ is trivial for any fiber of $\varphi$. Then there exists a vector bundle $X^{\prime}$ such that $\mathcal{E}=\varphi^*\mathcal{E}^{\prime}$.
\end{theorem}

\begin{proof}[Proof of Lemma \ref{blowup}]
By virtue of theorem \ref{andreatta}, we only have to show that $\mathcal{E}$ is trivial on the fiber F of $\varphi$. In this case F is isomorphic to $\mathbb{P}^{s-1}$ where $s$ is the codimension of Y in $X^{\prime}$. From Proposision \ref{line} we know that $\mathcal{E}|_F$ is uniform bundle on F. Therefore we have $\mathcal{E}|_F \cong \displaystyle \oplus ^{r} \mathcal{O}_{\mathbb{P}^{s-1}}(a)$. For two distinct fibers $F_1$ and $F_2$ of $\varphi$, we have $c_1(\mathcal{E} |_{F_1})=c_1(\mathcal{E}) |_{F_1}=c_1(\mathcal{E}) |_{F_2}=c_1(\mathcal{E} |_{F_2})$. Therefore $a$ is independent from the choice of F. 
Hence there exist a vector bundle $\mathcal{E}^{\prime}$ on $X^{\prime}$ and a integer k such that $\mathcal{E}\otimes \mathcal{O}_X(aE)=\varphi^*\mathcal{E}^{\prime}$ where E is the exceptional divisor of $\varphi$.
In this case we have $\Lambda_{\mathcal{E}}=\varphi^{*}\Lambda_{\mathcal{E}^{\prime}}$. 
Since  $\Lambda_{\mathcal{E}}$ is nef and $\varphi$ is surjective, $\Lambda_{\mathcal{E}^{\prime}}$ is also nef.
\end{proof}

\begin{proof}[Proof of Theorem \ref{positive}]
Since S is rational, we have a birational map $f : S \dashrightarrow \mathbb{P}^2$.
Let
 \[\xymatrix{
& X \ar[rd]^p \ar[ld]_q& \\  
S \ar@{-->}[rr]_{\varphi} &  & \mathbb{P}^2  \\
}\]
be a resolution of indeterminacy of $\varphi$.
By assumption we can show that $\Lambda_{q^{*}\mathcal{E}}$ is nef. Since p is a composition of blow-ups of a point, we can use Theorem \ref{blowup} and we get 
a vector bundle $\mathcal{E}^{\prime}$ on $\mathbb{P}^2$ such that $q^{*}\mathcal{E}\otimes \mathcal{L}=p^*\mathcal{E}^{\prime}$
for some line bundle $\mathcal{L}$ on X.
Moreover $\Lambda_{\mathcal{E}^{\prime}}$ is nef.
For any line $l$ in $\mathbb{P}^2$ we have a decomposition $\mathcal{E}^{\prime}|_l \cong \bigoplus\mathcal{O}^r_l$ by the argument as in proof of Theorem \ref{rc}.
In particular $\mathcal{E}^{\prime}$ is a uniform vector bundle on $\mathbb{P}^2$.
Hence $\mathcal{E}^{\prime}$ is isomorphic to a trivial vector bundle. Therefore $q^{*}\mathcal{E}\otimes \mathcal{L}=p^*\mathcal{E}^{\prime}$ is a trivial vector bundle on X.
From Lemma \ref{lemma} we know that $\mathcal{E}$ is trivial up to twist by a line bundle on S.
\end{proof}

If X is not rationally connected, even if X is uniruled, there is a non-trivial vector bundle $\mathcal{E}$ such that the normalized tautological divisor is nef. 

\begin{example}$($char$\;\geqslant0)$
Let C be a smooth elliptic curve and $Y=C\times \mathbb{P}^1$. 
Then Y is uniruled and $h^1(Y, \mathcal{O}_Y)=h^1(C, \mathcal{O}_C)=1$. 
Let $\mathcal{E}$ be a nontrivial extension of trivial line bundles $0 \rightarrow \mathcal{O}_C \rightarrow \mathcal{E} \rightarrow \mathcal{O}_C \rightarrow 0$. 
Then $\Lambda_{\mathcal{E}}=\xi_{\mathcal{E}}$ is nef. 
\end{example}

However we get following results. 

\begin{theorem}$($char$\;\geqslant0)$
Let $\varphi : X=\mathbb{P}_Y(\mathcal{F}) \rightarrow Y$ be a $\mathbb{P}^d$-bundle on a smooth projective variety Y, $\mathcal{E}$ a vector bundle of rank r on X and $\pi : \mathbb{P}_X(\mathcal{E}) \rightarrow X$ the natural projection. Assume the normalized tautological divisor $\Lambda_{\mathcal{E}}$ is nef. Then there exists a vector bundle $\mathcal{E}^{\prime}$ on Y such that $\mathcal{E}=\varphi^*\mathcal{E}^{\prime}$ such that $\Lambda_{\mathcal{E}^{\prime}}$ is nef.
\end{theorem}

\begin{proof}
Let F be a fiber of $\pi$. Then we have $\mathcal{E}|_F \cong \displaystyle \bigoplus^{r} \mathcal{O}_{\mathbb{P}^{d}}(a)$ where $a$ is an integer. Since $a$ is independent from 
the choice of the fiber F. Therefore $\mathcal{E}\otimes \mathcal{O}(-a\xi_{\mathcal{F}})=\pi^*\mathcal{E}^{\prime}$ where $\mathcal{E}^{\prime}$ is a rank r vector bundle on Y.
Because $\pi^{\prime} : \mathbb{P}_X(\mathcal{E}) \rightarrow \mathbb{P}_Y(\mathcal{E}^{\prime})$ is surjective and 
$\Lambda_{\mathcal{E}}={\pi^{\prime}}^*\Lambda_{\mathcal{E}^{\prime}}$, we have $\Lambda_{\mathcal{E}^{\prime}}$ is nef.
\end{proof}

\section{nefness of normalized tautological divisor of tangent bundle}
In this section we consider the nefness of the normalized tautological divisor of tangent bundle.

\begin{proposition}$($char$\;\geqslant0)$\label{nonnef}
Let X be a smooth projective variety which contains a rational curve $f : \mathbb{P}^1 \rightarrow X$. 
Then the normalized tautological divisor of tangent bundle $\Lambda_{\mathcal{T}_X}$ of X is not nef.
\end{proposition}

\begin{proof}
By Lemma \ref{line} we know that $f^*(\mathcal{T}_X)\cong\mathcal{O}(a)^{\oplus n}$ for some integer $a\geqslant2$ for any rational curve $f : \mathbb{P}^1 \rightarrow X$. 
Therefore we have $deg(f^*(-K_X))=na>0$. In particular $K_X$ is not nef.
If $K_X$ is not nef, then we can find a rational curve $g : \mathbb{P}^1 \rightarrow X$ on X such that $deg(g^*(-K_X))\leqslant n+1$ by Theorem 1.13 in \cite{kollar-mori}. 
This is a contradiction.
\end{proof}

From this Proposition, we immediately prove the following result.
\begin{corollary}$($char$\;\geqslant0)$\label{minimal}
Let X be a smooth projective variety over algebraically closed field $k=\bar{k}$ of arbitrary characteristic. 
If $K_X$ is not nef then the normalized tautological divisor of tangent bundle $\Lambda_{\mathcal{T}_X}$ is not nef.
\end{corollary}

\begin{corollary}$($char$\;\geqslant0)$
Let S be a smooth projective surface over algebraically closed field $k=\bar{k}$ of arbitrary characteristic of negative Kodaira dimension. 
Then the normalized tautological divisor of tangent bundle $\Lambda_{\mathcal{T}_S}$ is not nef.
\end{corollary}

In the case where the canonical divisor is numerical trivial, the nefness of the normalized tautological divisor of tangent bundle is equivalent to the nefness of its tangent bundle.
F.Campana and T.Peternell proved that in this case there is an $\acute{e}tale$ covering $T \rightarrow X$ from abelian variety A.

\begin{theorem}[\cite{campana} Theorem 2.3]\label{campana}
Let X be a smooth projectie manifold. Assume that the tangent bundle $\mathcal{T}_X$ is nef and $K_X$ is nef,
then the canonical bundle is numerical trivial $K_X\equiv 0$ and there is an $\acute{e}tale$ covering $A \rightarrow X$ from abelian variety A.
\end{theorem}

\begin{corollary}
Let S be a smooth projective complex surface of Kodaira dimension zero. If the normalized tautological divisor of tangent bundle is nef, there is an $\acute{e}tale$ covering $A \rightarrow S$ from abelian surface A.
\end{corollary}

The case where Kodaira dimension $\kappa(X)\geqslant 1$ is very difficult in general.
But the case where $dimX=2$, we can obtain the following result. 

\begin{proposition}
Let S be a smooth projective minimal complex surface with the Kodaira dimension $\kappa(S)\geqslant1$. Then the normalized tautological divisor of tangent bundle 
$\Lambda_{\mathcal{T}_S}$ is not nef.
\end{proposition}

\begin{proof}
If the Kodaira dimension $\kappa(S)=1$, then S is a relatively minimal elliptic surface $\pi : S \rightarrow C$. By virtue of Theorem \ref{nakayama} if $\Lambda_{\mathcal{T}_S}$ is nef we have $c_1(S)^2=c_2(S)=0$.
Therefore we have the Euler character $\chi(\mathcal{O}_S)=0$ and a singular fiber of S is multiple fiber from corollary 16 and 17 in \cite{friedman}.
By canonical bundle formula for elliptic surfaces $($c.f Theorem 15 in \cite{friedman}$)$ we obtain
\[K_S=\pi^*(K_C\otimes \mathcal{L})\otimes \mathcal{O}_S(\displaystyle \sum_i(m_i-1)F_i)\]
where $\mathcal{L}$ is a line bundle on C of $deg(\mathcal{L})=0$ and $F_i$ is the multiple fiber of $\pi$ with the multiplicity $m_i$.
We consider the exact sequence of sheaves $0 \rightarrow \pi^*(\Omega_C) \rightarrow \Omega_S$.
This sheaf morphism drops rank on multible fibers. Therefore we can get the torsion-free subsheaf  $0 \rightarrow \pi^*(\Omega_C)\otimes \mathcal{O}_S(\displaystyle \sum_i(m_i-1)F_i) \rightarrow \Omega_S$.
In this case we have $\pi^*(\Omega_C)\otimes \mathcal{O}_S(\displaystyle \sum_i(m_i-1)F_i)\equiv_{num} K_S$.
Since $K_S$ is not numerical trivial, for an ample divisor A we have $\mu_A(\pi^*(\Omega_C)\otimes \mathcal{O}_S(\displaystyle \sum_i(m_i-1)F_i))=A.K_S>A.K_S/2=\mu_A(\Omega_S)$.
This is a contradiction to the semistability of $\Omega_S$.

If the Kodaira dimension $\kappa(S)=2$ i.e. S is of general type, we have $0<c_1^2(S)\leqslant 3c_2(S)<4c_2(S)=c_1^2(S)$ by Theorem \ref{nakayama} and Miyaoka-Yau inequality. It is a contradiction.
\end{proof}

Combining with these results we have a characterization of abelian surface up to finite \'etale covering. 

\begin{theorem}
Let S be a smooth projective complex surface. 
If the normalized tautological line bundle of tangent bundle is nef, then there is an $\acute{e}tale$ covering $A \rightarrow S$ from abelian surface A.
\end{theorem}

In higher dimensional case we have the following characterization of abelian varieties up to finite \'etale covering.
The following result is due to Eiichi Sato.

\begin{theorem}$($E. Sato$)$
Let X be a smooth projective variety of n dimensional. 
If the normalized tautological divisor of tangent bundle is semiample, then there is an $\acute{e}tale$ covering $A \rightarrow X$ from abelian variety A.
\end{theorem}
\begin{proof}
By virtue of Theorem \ref{nonnef} we may assume that X is minimal. 
From Theorem \ref{semiample} we have a finite \'etale covering $f: \tilde{X}\rightarrow X$ such that $\mathcal{T}_{\tilde{X}} \cong f^*\mathcal{T}_X\cong \mathcal{L}^{\oplus n}$ for some line bundle $\mathcal{L}$ on $\tilde{X}$.
If $\kappa(X)\geqslant1$ we have $K_{\tilde{X}}.A^{n-1}=-n\mathcal{L}.A^{n-1}>0$ for an ample divisor on $\tilde{X}$.
By Proposition 2 in \cite{druel} we know that the universal covering space of X is isomorphic to the direct product $D^n$ where D is an open disc $D=\{z\in\mathbb{C}| |z|<1\}$.
It contradicts to the decomposition into same line bundles $\mathcal{T}_{\tilde{X}}\cong \mathcal{L}^{\oplus n}$.
Therefore we may consider only the case where $\kappa(X)\leqslant0$.
In this case there is an $\acute{e}tale$ covering $T \rightarrow X$ from abelian variety T by Corollary \ref{minimal} and Theorem \ref{campana}.
\end{proof}

\end{document}